\documentclass[oneside,english]{amsart}
\usepackage[LGR,T1]{fontenc}
\usepackage[latin9]{inputenc}
\usepackage{amstext}
\usepackage{amsthm}
\usepackage{amssymb}

\makeatletter

\DeclareRobustCommand{\greektext}{%
  \fontencoding{LGR}\selectfont\def\encodingdefault{LGR}}
\DeclareRobustCommand{\textgreek}[1]{\leavevmode{\greektext #1}}

\numberwithin{equation}{section}
\numberwithin{figure}{section}
\theoremstyle{plain}
\newtheorem{thm}{\protect\theoremname}[section]
\theoremstyle{definition}
\newtheorem{defn}[thm]{\protect\definitionname}
\theoremstyle{plain}
\newtheorem{lem}[thm]{\protect\lemmaname}
\theoremstyle{plain}
\newtheorem{prop}[thm]{\protect\propositionname}

\makeatother

\usepackage{babel}
\providecommand{\definitionname}{Definition}
\providecommand{\lemmaname}{Lemma}
\providecommand{\propositionname}{Proposition}
\providecommand{\theoremname}{Theorem}

\begin{document}
\title[{\tiny a study on product of filter large sets and related structures}]{A STUDY ON PRODUCT OF FILTER LARGE SETS AND RELATED STRUCTURES}
\author{Jyotirmoy Poddar, Sujan Pal}
\address{{\large Department of Mathematics, Techno India University West Bengal,
Kolkata-700091, West Bengal, India.}}
\email{{\large jyotirmoy.p@technoindiaeducation.com}}
\address{{\large Department of Mathematics, University of Kalyani, Kalyani-741235,
Nadia, West Bengal, India.}}
\email{{\large sujan2016pal@gmail.com}}
\keywords{Central sets, Combinatorially large sets, Algebra of Stone-\v{C}ech
compactification of descrete semigroup, Idempotent filters.}
\begin{abstract}
Sets satisfying Central sets theorem and other Ramsey theoretic large
sets were studied extensively in literature. Hindman and Strauss proved
that product of some of these large sets is again large. In this paper
we show that if we take two combinatorially large sets along idempotent
filters, then their product is also a filter large set. Also we show
that the product of two Strongly Central sets and Thickly Central
sets are also combinatorially large. The techniques we use are majorly
combinatorial in nature.
\end{abstract}

\maketitle

\section{Introduction}

The first idea of partition in Ramsey theory we'll discuss arises
from van der Waerden's Theorem \cite{key-15}, which was published
in 1927. The theorem says that if we colour the set of natural numbers
$\mathbb{N}$ finitely, then there exists one cell which contains
arithmetic progression of arbitrary length. On the other hand Ramsey
theory deals with infinite sets. Although it originated in combinatorics,
many applications in Ramsey theory has been found in recent years
and can now be understood via the viewpoints of analysis, algebra,
and number theory. Here is the statement of Ramsey's theorem where
$\left[S\right]^{k}$ is the collection of all $k$ cardinality subsets
of the set $S$. 
\begin{thm}
(Ramsey) If we take the set of all two element subsets of $\mathbb{N}$,
i.e, $\left[\mathbb{N}\right]^{2}=\left\{ \left\{ n,m\right\} :n,m\in\mathbb{N}\right\} $
and we finitely color $\left[\mathbb{N}\right]^{2}$, then there exists
some infinite set $E\subseteq\mathbb{N}$ such that $\left[E\right]^{2}$
is monocromatic.
\end{thm}

Several concepts of largeness for subsets of a semigroup $\left(S,\cdot\right)$
will be of our interest. All of these notions are closed under passage
to supersets. Central sets have a very rich literature in Ramsey theory.
These are the large sets which satisfies Central Sets Theorem, originally
proved by Furstenberg \cite{key-5} and later its various versions
by other mathematicians. For over a hundred years there have been
many devlopements, specially towards characterization of Central sets
and Central sets theorem, firstly by Furstenberg himself and then
by Hindman, Strauss, De and many others in \cite{key-6,key-7,key-8,key-9,key-5,key-3}.
Not only Central sets but there are other combinatorially large sets
like $IP$ sets, $J$ sets and many others which mathematicians like
to study from Ramsey theoretic aspects. There are several techniques
that mathematicians use to study the large sets, namely the algebra
of Stone-\v{C}ech compactification, Ergodic theory or basic combinatorial
tools.

Making a new large set can be done in a variety of ways. It is possible
to verify if the product of two large sets is indeed a large set.
It can also be used to make new large sets out of the old ones. In
\cite{key-11} Hindman used technique of algebra to show that product
of two similar type of Ramsey theoretic large sets is again large.
Goswami gave a combinatorial proof of these facts in \cite{key-6}.
In this paper we want to take the results further for a more generalised
setting and our approach is purely combinatorial.

In the next section we give several important concepts and definitions
from the literature, namely the concepts of idempotent filters on
a semigroup which was first introduced in \cite{key-14} and then
further developments done in \cite{key-7,key-2-1}. Then in section
3 we show that filter version of syndetic sets and Piecewise syndetic
are closed under cartesian product. In the fourth we proved that the
same property also hold for filter version of $J$-sets. The product
of a few large sets in $\beta S$ those has not been looked at yet
is covered in the last section.

\section{Preliminaries}

We start with the basic definition of filters on a set $S$.
\begin{defn}
Let $S$ be any set. Let $\mathcal{U}$ be a non-empty set of subsets
of $S$. $\mathcal{U}$ is called a \emph{filter} on $S$ if it satisfies
the following properties:

\begin{enumerate}
\item If $A,\,B\in\mathcal{U}$, then $A\cap B\in\mathcal{U}$;
\item If $A\in\mathcal{U}$ and $A\subseteq B\subseteq S$, then $B\in\mathcal{U}$;
\item $\emptyset\notin\mathcal{U}$.
\end{enumerate}
\end{defn}

A classic example of a filter is the set of neighborhoods of a point
in a topological space.
\begin{defn}
An \emph{ultrafilter }on $S$ is a maximal filter on $S$. That is
an \emph{ultrafilter }on $S$ is itself a filter on $S$ which is
not contained properly in any other filter on $S$. Let $S$ be any
set, and let $a$ be an element of $S$. Then the collection of sets
each of which contains $a$ is said to be a \emph{principal }ultrafilter
corresponding to $a\in S$. In fact the principal ultrafilters are
the only ones whose members can be explicitly defined. 

We now give a brief review about the Stone-\v{C}ech compactification
of a discrete semigroup. Let $\left(S,\cdot\right)$ be any discrete
semigroup and denote its Stone-\v{C}ech compactification by $\beta S$.
$\beta S$ is the set of all ultrafilters on $S$, where the points
of $S$ are identified with the principal ultrafilters. The basis
for the topology is $\left\{ \bar{A}:A\subseteq S\right\} $, where
$\bar{A}=\left\{ p\in\beta S:A\in p\right\} $. The operation of $S$
can be extended to $\beta S$ making $\left(\beta S,\cdot\right)$
a compact, right topological semigroup containing $S$ in its topological
center. That is, for all $p\in\beta S$, the function $\rho_{p}:\beta S\rightarrow\beta S$
is continuous, where $\rho_{p}\left(q\right)=q\cdot p$ and for all
$x\in S$, the function $\lambda_{x}:\beta S\rightarrow\beta S$ is
continuous, where $\lambda_{x}\left(q\right)=x\cdot q$. For $p,q\in\beta S$
and $A\subseteq S$, $A\in p\cdot q$ if and only if $\left\{ x\in S:x^{-1}A\in q\right\} \in p$,
where $x^{-1}A=\left\{ y\in S:x\cdot y\in A\right\} $. 
\end{defn}

Since $\beta S$ is a compact Hausdorff right topological semigroup,
it has a smallest two sided ideal denoted by $K\left(\beta S\right)$,
which is the union of all of the minimal right ideals of $S$, as
well as the union of all of the minimal left ideals of $S$. Every
left ideal of $\beta S$ contains a minimal left ideal and every right
ideal of $\beta S$ contains a minimal right ideal. The intersection
of any minimal left ideal and any minimal right ideal is a group,
and any two such groups are isomorphic. Any idempotent $p$ in $\beta S$
is said to be minimal if and only if $p\in K\left(\beta S\right)$.
Though Central sets was defined dynamically, there is an algebraic
counterpart of this definition, established by V. Bergelson and N.
Hindman in \cite{key-2-1}. For the sake of our work we need to revisit
some important definitions. For more details see \cite{key-12}.
\begin{defn}
Let $\left(S,\cdot\right)$ be a semigroup and $A\subseteq S$, then
\end{defn}

\begin{enumerate}
\item The set $A$ is thick if and only if for any finite subset $F$ of
$S$, there exists an element $x\in S$ such that $F\cdot x\subset A$.
This means the sets which contains a translation of any finite subset.
For example one can see $\cup_{n\in\mathbb{N}}\left[2^{n},2^{n}+n\right]$
is a thick set in $\mathbb{N}$.
\item The set $A$ is syndetic if and only if there exists a finite subset
$G$ of $S$ such that $\bigcup_{t\in G}t^{-1}A=S$. That is, with
a finite translation if the set covers the entire semigroup, then
it will be called a Syndetic set. For example the set of even and
odd numbers are both syndetic in $\mathbb{N}$.
\item The sets which can be written as an intersection of a syndetic and
a thick set are called Piecewise syndetic sets. More formally a set
$A$ is Piecewise syndetic if and only if there exists $G\in\mathcal{P}_{f}\left(S\right)$
such that for every $F\in\mathcal{P}_{f}\left(S\right)$, there exists
$x\in S$ such that $F\cdot x\subseteq\bigcup_{t\in G}t^{-1}A$. Clearly
the thick sets and syndetic sets are natural examples of Piecewise
syndetic sets. From definition one can immediately see that $2\mathbb{N}\cap\bigcup_{n\in\mathbb{N}}\left[2^{n},2^{n}+n\right]$
is a nontrivial example of Piecewise syndetic sets in $\mathbb{N}$.
\item $\mathcal{T}=\,^{\mathbb{N}}S$.
\item For $m\in\mathbb{N}$, $\mathcal{J}_{m}=\left\{ \left(t\left(1\right),\ldots,t\left(m\right)\right)\in\mathbb{N}^{m}:t\left(1\right)<\ldots<t\left(m\right)\right\} .$
\item Given $m\in\mathbb{N}$, $a\in S^{m+1}$, $t\in\mathcal{J}_{m}$ and
$f\in F$, 
\[
x\left(m,a,t,f\right)=\left(\prod_{j=1}^{m}\left(a\left(j\right)\cdot f\left(t\left(j\right)\right)\right)\right)\cdot a\left(m+1\right)
\]
where the terms in the product $\prod$ are arranged in increasing
order.
\item $A\subseteq S$ is called a $J$-set iff for each $F\in\mathcal{P}_{f}\left(\mathcal{T}\right)$,
there exists $m\in\mathbb{N}$, $a\in S^{m+1}$, $t\in\mathcal{J}_{m}$
such that, for each $f\in\mathcal{T}$,
\[
x\left(m,a,t,f\right)\in A.
\]
\item If the semigroup $S$ is commutative, the definition is rather simple.
In that case, a set $A\subseteq S$ is a $J$-set if and only if whenever
$F\in\mathcal{P}_{f}\left(^{\mathbb{N}}S\right)$, there exist $a\in S$
and $H\in\mathcal{P}_{f}\left(\mathbb{N}\right)$, such that for each
$f\in F$, $a+\sum_{t\in H}f(t)\in A$.
\item If we are given any injective sequence $\langle x_{n}\rangle_{n=1}^{\infty}$
in $S$, then, a set $A$ which contains $FP\left(\langle x_{n}\rangle_{n=1}^{\infty}\right)$
for some injective sequence $\langle x_{n}\rangle_{n=1}^{\infty}$
in $S$, is called an IP set, where 
\[
FP\left(\langle x_{n}\rangle_{n=1}^{\infty}\right)=\left\{ x_{i_{1}}\cdot x_{i_{2}}\cdot\cdots\cdot x_{i_{n}}:\left\{ i_{1}<i_{2}<\cdots<i_{n}\right\} \subseteq\mathbb{N}\right\} .
\]
\item Then a subset $A$ of $S$ is called central if and only if there
is some minimal idempotent $p$ such that $A\in p$.
\end{enumerate}
Since $K\left(\beta S\right)$ can be expressed as union of minimal
left or right ideals, then it becomes natural to ask whether there
exists sets which does not meet every minimal left ideal in some idempotent.
The answer of the question turns out to be yes. This proposed a new
notion of large sets for semigroup. 
\begin{defn}
Let $S$ be a discrete semigroup and let $C$ be a subset of $S$.
Then $C$ is said to be \textit{strongly central} if for every minimal
left ideal $L$ of $\beta S$, $\overline{C}\cap L$ contains an idempotent.
\end{defn}

This definition first appeared in \cite{key-1-1} where the authors
also provided dynamical characterization of \emph{strongly central
sets.}

We will now discuss a notion which is the turning point towards the
topics of our current work. This concept first defined in \cite{key-14}.
Throughout this paper, $\mathcal{F}$will denote a filter of $\left(S,\cdot\right)$.
For every filter $\mathcal{F}$of $S$, define $\bar{\mathcal{F}}\subseteq\beta S$,
by
\[
\bar{\mathcal{F}}=\bigcap_{V\in\mathcal{F}}\bar{V}.
\]
 It is a routine check that $\bar{\mathcal{F}}$ is a closed subset
of $\beta S$ consisting of ultrafilters which contain $\mathcal{F}$.
If $\mathcal{F}$ is an idempotent filter, i.e., $\mathcal{F}\subset\mathcal{F}\cdot\mathcal{F}$,
then $\bar{\mathcal{F}}$ becomes a closed subsemigroup of $\beta S$,
but the converse is not true. Throughout our article, we will consider
only those filters $\mathcal{F}$, for which $\bar{\mathcal{F}}$
is a closed subsemigroup of $\beta S$.

In light of this notion we can define the concept of piecewise $\mathcal{F}$-syndeticity
both combinatorially and algebraically. For details see \cite{key-14}.
\begin{defn}
Let $T$ be a closed subsemigroup of $\beta S$ and $\mathcal{F}$
be a filter on $S$ such that $\bar{\mathcal{F}}=T$.

(1) A subset $A$ of $S$ is $\mathcal{F}$-syndetic if for every
$V\in\mathcal{F}$, there is a finite set $G\subseteq V$ such that
$G^{-1}A\in\mathcal{F}$.

(2) A subset $A\subseteq S$ is piecewise $\mathcal{F}$-syndetic
if for every $V\in\mathcal{F}$, there is a finite $F_{V}\subseteq V$
and $W_{V}\in\mathcal{F}$ such that whenever $H\subseteq W_{V}$
a finite subset, there is $y\in V$ such that $H\cdot y\subseteq F_{V}^{-1}A$.
\end{defn}

Here is an algebraic characterization of piecewise $\mathcal{F}$-syndetic
sets.
\begin{thm}
Let $T$ be a closed subsemigroup of $\beta S$, and $\mathcal{F}$be
the filter on $S$ such that $T=\bar{\mathcal{F}}$, also let $A\subseteq S$.
Then $\bar{A}\cap K\left(T\right)\neq\emptyset$ if and only if $A$
is piecewise $\mathcal{F}$-syndetic.
\end{thm}

\begin{proof}
See \cite{key-14}.
\end{proof}
Being motivated from these concepts the authors in \cite{key-7} defined
the concepts of Definition 2.3 in this new framework which we will
require in this paper.
\begin{defn}
Let $\left(S,\cdot\right)$ be an arbitrary semigroup and $\mathcal{F}$
be a filter on $S$. Then 
\end{defn}

\begin{enumerate}
\item For any $l\in\mathbb{N}$, and any $l$-sequences $\langle x_{n}^{(i)}\rangle_{n=1}^{\infty}$
for $i\in\left\{ 1,2,\cdots,l\right\} $, define the zigzag finite
product 
\[
ZFP\left(\langle x_{n}^{(i)}\rangle_{i,n=1,1}^{l,\infty}\right)=\left\{ \begin{array}{c}
\prod_{t\in H}y_{t}:H\in\mathcal{P}_{f}\left(\mathbb{N}\right)\,\text{and}\\
y_{i}\in\left\{ x_{i}^{(1)},x_{i}^{(2)},\cdots,x_{i}^{(l)}\right\} \,\text{for\,any}\,i\in\mathbb{N}
\end{array}\right\} .
\]
\item For any $k\in\mathbb{N}$, define 
\[
ZFP_{k}\left(\langle x_{n}^{(i)}\rangle_{i,n=1,1}^{l,\infty}\right)=\left\{ \begin{array}{c}
\prod_{t\in H}y_{t}:H\in\mathcal{P}_{f}\left(\mathbb{N}\right)\,\text{,min\ensuremath{H\geq k\,}and}\\
y_{i}\in\left\{ x_{i}^{(1)},x_{i}^{(2)},\cdots,x_{i}^{(l)}\right\} \,\text{for\,any}\,i\in\mathbb{N}
\end{array}\right\} .
\]
\item Let $G\in\mathcal{P}_{f}\left(^{\mathbb{N}}S\right)$, we will call
$G$ is $\mathcal{F}$ -good if for any $F\in\mathcal{F}$, there
exists $k=k\left(F\right)\in\mathbb{N}$ such that $ZFP_{k}\left(G\right)\subseteq F$. 
\item A set $B\subseteq S$ will be called a $\mathcal{F}$-$J$ set, if,
for any $\mathcal{F}$ -good map, say $F\in\mathcal{P}_{f}\left(^{\mathbb{N}}S\right)$,
there exist $a_{1},a_{2},\ldots,a_{m+1}\in S$ and $\left\{ h_{1},h_{2}\cdots,h_{m}\right\} _{\leq}\subset\mathbb{N}$
such that 
\[
x\left(m,a,h,f\right)=a_{1}f\left(h_{1}\right)a_{2}f\left(h_{2}\right)\ldots a_{m}f\left(h_{m}\right)a_{m+1}\in B.
\]
\item $\mathcal{P}_{f}^{\mathcal{F}}\left(^{\mathbb{N}}S\right)=\left\{ \begin{array}{ccccc}
F\in\mathcal{P}_{f}\left(^{\mathbb{N}}S\right): & F & \text{\,is} & \mathcal{F} & \text{\,-good}\end{array}\right\} .$
\item A set $A\subseteq S$ is called $\mathcal{F}$-central if and only
if there exists an idempotent ultrafilter $p\in K\left(\bar{\mathcal{F}}\right)$
such that $A\in p$.
\end{enumerate}
Now we need to establish the concept of product filter. Let $\mathcal{F}$
and $\mathcal{G}$ be two given filters on two semigroups $S$ and
$T$ respectively, such that $\bar{\mathcal{F}}$ and $\bar{\mathcal{G}}$
are two closed subsemigroups. Let $\mathcal{H}$ be a filter generated
by $\mathcal{F}\times\mathcal{G}$ over $S\times T$, i.e., $\mathcal{H}=\left\{ D:D\supset A\times B\,\text{for\,some}\,A\in\mathcal{F},B\in\mathcal{G}\right\} $.
We will consider those $\mathcal{H}$ which generates a closed subsemigroup
in $\beta\left(S\times T\right)$. The following lemma establishes
the existence of such filter on the product space.
\begin{lem}
If $\mathcal{F}$and $\mathcal{G}$ are two idempotent filters on
$S$ and $T$ respectively, then $\mathcal{H}$ is an idempotent filter
on $S\times T$ and hence $\bar{\mathcal{H}}$ is closed subsemigroup
on $S\times T$.
\end{lem}

\begin{proof}
As $\mathcal{F}$and $\mathcal{G}$ are idempotent filters, we have
$\mathcal{F}\subset\mathcal{F}\cdot\mathcal{F}$ and $\mathcal{G}\subset\mathcal{G}\cdot\mathcal{G}$.
Let $A\in\mathcal{H}$, hence there is $B\in\mathcal{F}$ and $C\in\mathcal{G}$
such that $B\times C\subset A$. So, $\left\{ x:x^{-1}B\in\mathcal{F}\right\} \in\mathcal{F}$
and $\left\{ y:y^{-1}C\in\mathcal{G}\right\} \in\mathcal{G}$ and
this implies 
\[
\left\{ \left(x,y\right):\left(x,y\right)^{-1}\left(B\times C\right)\in\mathcal{F}\times\mathcal{G}\right\} \in\mathcal{F}\times\mathcal{G}
\]
Hence $\left\{ \left(x,y\right):\left(x,y\right)^{-1}A\in\mathcal{H}\right\} \in\mathcal{H}$.
Hence $\mathcal{H}$ is an idempotent filter and so $\bar{\mathcal{H}}$
is a closed subsemigroup on $S\times T$.
\end{proof}

\section{Product of $\mathcal{F}-$ Syndetic and Piecewise $\mathcal{F}-$
Syndetic Sets}

The following theorem shows that the product of two filter large syndetic
sets is again filter large syndetic.
\begin{thm}
If $\mathcal{F}$and $\mathcal{G}$ are two idempotent filters on
$S$ and $T$ respectively such that $\bar{\mathcal{F}}$and $\bar{\mathcal{G}}$are
two closed subsemigroups and $\mathcal{H}$ is a filter on $S\times T$
generated by $\mathcal{F}$and $\mathcal{G}$ such that $\bar{\mathcal{H}}$is
a closed subsemigroup of $\beta\left(S\times T\right)$. If $A$ and
$B$ are $\mathcal{F}$-syndetic and $\mathcal{G}$-syndetic sets
in $S$ and $T$ respectively then $A\times B$ is an $\mathcal{H}$-syndetic
set.
\end{thm}

\begin{proof}
This proof is a two liner. Let $V\in\mathcal{H}$, so there exists
$C\in\mathcal{F}$and $D\in\mathcal{G}$ such that $C\times D\subseteq V$.
By definition, there exists finite sets $F\subset C$ and $G\subset D$
such that $F^{-1}A\in\mathcal{F}$ and $G^{-1}B\in\mathcal{G}$. Hence
\[
\left(F\times G\right)^{-1}\left(A\times B\right)\in\mathcal{F}\times\mathcal{G}\subset\mathcal{H}
\]
 and so $A\times B$ is $\mathcal{H}$-syndetic.
\end{proof}
For our proof of product of filter piecewise syndetic sets, we need
an equivalent definition, different from the one given in the previous
section in Def. 2.5, which was first mentioned in \cite{key-14}.
\begin{defn}
Let $T$ be a closed subsemigroup of $\beta S$ and $\mathcal{F}$be
a filter on $S$ such that $\bar{\mathcal{F}}=T$ and let $A\subseteq S$.
$A$ is called piecewise $\mathcal{F}$-syndetic if for every $V\in\mathcal{F}$,
there is a finite $F_{V}\subseteq V$ and $W_{V}\in\mathcal{F}$ such
that the family 
\[
\left\{ \left(x^{-1}F_{V}^{-1}A\right)\cap V:V\in\mathcal{F},x\in W_{V}\right\} 
\]
 has the finite intersection property.
\end{defn}

Now we are in a position to state the theorem.
\begin{thm}
If $\mathcal{F}$and $\mathcal{G}$ are two idempotent filters on
$S$ and $T$ respectively such that $\bar{\mathcal{F}}$and $\bar{\mathcal{G}}$are
two closed subsemigroups and $\mathcal{H}$ is a filter on $S\times T$
generated by $\mathcal{F}$and $\mathcal{G}$ such that $\bar{\mathcal{H}}$is
a closed subsemigroup of $\beta\left(S\times T\right)$. If $A$ and
$B$ are piecewise $\mathcal{F}$-syndetic and piecewise $\mathcal{G}$-syndetic
sets in $S$ and $T$ respectively then $A\times B$ is an piecewise
$\mathcal{H}$-syndetic set.
\end{thm}

\begin{proof}
For every $V\in\mathcal{H}$ there exists $C\in\mathcal{F}$ and $D\in\mathcal{G}$
such that $C\times D\subseteq V$. Hence from the definition of filter
piecewise syndeticity, there exists finite $F_{C}\subset C$, $F_{D}\subset D$
and $W_{C}\in\mathcal{F}$, $W_{D}\in\mathcal{G}$ such that, $\left\{ \left(x^{-1}F_{C}^{-1}A\right)\cap C:C\in\mathcal{F},x\in W_{C}\right\} $
and $\left\{ \left(y^{-1}F_{D}^{-1}B\right)\cap D:D\in\mathcal{G},y\in W_{D}\right\} $
both have the finite intersection property. Hence 
\[
\left\{ \left(\left(x,y\right)^{-1}\left(F_{C}\times F_{D}\right)^{-1}\left(A\times B\right)\right)\cap\left(C\times D\right):\left(x,y\right)\in W_{C}\times W_{D},C\times D\in\mathcal{F}\times\mathcal{G}\right\} 
\]

has the finite intersection property. This implies 
\[
\left\{ \left(\left(x,y\right)^{-1}\left(F_{C}\times F_{D}\right)^{-1}\left(A\times B\right)\right)\cap V:\left(x,y\right)\in W_{C}\times W_{D},V\in\mathcal{H}\right\} 
\]

has the finite intersection property, hence $A\times B$ is piecewise
$\mathcal{H}$-syndetic.
\end{proof}

\section{Product of Filter J sets}

In this section we want to discuss what happens when we take the cartesian
product of two filter $J$ sets. Motivation of this work comes from
the work of Goswami \cite{key-6}, in which he proved two lemmas which
were the basis of the main result. Here we prove the filter analogue
of those lemmas. Before entering in the results we want to give a
definition from \cite{key-12} which we will require to prove the
results of this section.
\begin{defn}
Let $\left(S,\cdot\right)$ be a semigroup and $\left\langle x_{n}\right\rangle _{n=1}^{\infty}$
be a sequence in $S$. The sequence $\left\langle y_{n}\right\rangle _{n=1}^{\infty}$
is called a product subsystem of $\left\langle x_{n}\right\rangle _{n=1}^{\infty}$
if and only if there is a sequence $\left\langle H_{n}\right\rangle _{n=1}^{\infty}$
in $\mathcal{P}_{f}\left(\mathbb{N}\right)$ such that for every $n\in\mathbb{N}$,
$max\,H_{n}<min\,H_{n+1}$ and $y_{n}=\prod_{t\in H_{n}}x_{t}$. 
\end{defn}

Also we define another notion.
\begin{defn}
Let $\left\langle F_{n}\right\rangle _{n=1}^{\infty}$ be a sequence
in $\mathcal{P}_{f}\left(\mathbb{N}\right)$.

(a) $FU\left(\left\langle F_{n}\right\rangle _{n=1}^{\infty}\right)=\left\{ \bigcup_{n\in G}F_{n}:G\in\mathcal{P}_{f}\left(\mathbb{N}\right)\right\} $
.

(b) We call $\left\langle G_{n}\right\rangle _{n=1}^{\infty}$ a union
subsystem of $\left\langle F_{n}\right\rangle _{n=1}^{\infty}$ iff
there is a sequence $\left\langle H_{n}\right\rangle _{n=1}^{\infty}$
in $\mathcal{P}_{f}\left(\mathbb{N}\right)$ such that for every $n\in\mathbb{N}$,
$max\,H_{n}<min\,H_{n+1}$ and $G_{n}=\bigcup_{t\in H_{n}}F_{t}$.
\end{defn}

Now we will state and prove several results before entering into the
main theorem.
\begin{lem}
Let $\left(S,\cdot\right)$ be a semigroup, and $A$ be an $\mathcal{F-}J$-set
in $S$, let $F\in\mathcal{P}_{f}\left(^{\mathbb{N}}S\right)$ be
$\mathcal{F}-$good. Also let 
\[
\Lambda=\left\{ \begin{array}{c}
M\subset\mathbb{N}:M\,\text{is\,finite\,and}\,M=\left\{ \tau\left(1\right),\tau\left(2\right),\ldots,\tau\left(m\right)\right\} _{<}\\
\,\text{and}\,\,\exists a\in S^{m+1}\,\text{\,such\,that}\,\,\forall f\in F\\
\left(a\left(1\right)\cdot f\left(\tau\left(1\right)\right)\cdot a\left(2\right)\ldots a\left(m\right)\cdot f\left(\tau\left(m\right)\right)\cdot a\left(m+1\right)\in A\right)
\end{array}\right\} .
\]
 If $\left\langle G_{n}\right\rangle _{n=1}^{\infty}$ be a sequence
in $\mathcal{P}_{f}\left(\mathbb{N}\right)$ such that $maxG_{n}<minG_{n+1}$
for each $n\in\mathbb{N}$, then there exists $H\in\mathcal{P}_{f}\left(\mathbb{N}\right)$
such that $\bigcup_{n\in H}G_{n}\in\Lambda$.
\end{lem}

\begin{proof}
Let for each $n\in\mathbb{N}$, $Card\left(G_{n}\right)=\gamma_{n}$.
Let 
\[
G_{n}=\left\{ \ell_{n,1},\ell_{n,2},\cdots,\ell_{n,\gamma_{n}}\right\} _{<}.
\]
 We have that $F$ is an $\mathcal{F}-$good map. So by definition
we get, for a given $V\in\mathcal{F}$, we get a $k=k\left(V\right)\in\mathbb{N}$,
such that $ZFP_{k}\left(F\right)\subseteq V$.

Define 
\[
h_{f}\left(n\right)=f\left(\ell_{n,1}\right)\cdot f\left(K\right)\cdot f\left(\ell_{n,2}\right)\cdot f\left(K\right)\cdots f\left(K\right)\cdots f\left(\ell_{n,\gamma_{n}}\right)
\]
 where $K\geq k$.

For this $h_{f}$, the product on the right side is in $V$, since
$K\geq k$ and by definition $F$ is $\mathcal{F}-$good map. So 
\[
ZFP_{k}\left(\left\{ h_{f}:f\in F\right\} \right)\subseteq ZFP_{k}\left(F\right)\subseteq V.
\]

It means that if $F$ is an $\mathcal{F}-$good map, so is the collection
$\left\{ h_{f}:f\in F\right\} \in\mathcal{P}_{f}\left(^{\mathbb{N}}S\right)$.
So by definiton of $\mathcal{F}-$\emph{J }set, we have that for this
$\mathcal{F}-$good map $\left\{ h_{f}:f\in F\right\} $, we have
$m\in\mathbb{N}$, $a\in S^{m+1}$ and 
\[
\tau\left(1\right)<\tau\left(2\right)<\ldots<\tau\left(m\right)
\]
 from natural numbers such that 
\[
a\left(1\right)\cdot h_{f}\left(\tau\left(1\right)\right)\cdot a\left(2\right)\cdot h_{f}\left(\tau\left(2\right)\right)\cdots a\left(m\right)\cdot h_{f}\left(\tau\left(m\right)\right)\cdots a\left(m+1\right)\in A.
\]

We take $H=\left\{ \tau\left(1\right),\tau\left(2\right),\ldots,\tau\left(m\right)\right\} $.
This is our required set.
\end{proof}
\begin{prop}
Let an arbitrary semigroup be $\left(S,\cdot\right)$ and a filter
be $\mathcal{F}$. Let a sequence in $S$ be $\langle x_{n}\rangle_{n=1}^{\infty}$
and an idempotent in the form of $p$ in $\bar{\mathcal{F}}$ so that
given any $m\in\mathbb{N}$, $FP\left(\langle x_{n}\rangle_{n=m}^{\infty}\right)\in p$.
Let $A\in p$, so a product subsystem $\langle y_{n}\rangle_{n=1}^{\infty}$
of $\langle x_{n}\rangle_{n=1}^{\infty}$ can be obtained, so that
$FP\left(\langle y_{n}\rangle_{n=1}^{\infty}\right)\subseteq A$. 
\end{prop}

\begin{proof}
Recall $A^{\,*}=\left\{ x\in A:x^{-1}A\in p\right\} $. Pick $y_{1}\in A^{\,*}\cap FP\left(\langle x_{n}\rangle_{n=1}^{\infty}\right)$.
Also pick $H_{1}\in\mathcal{P}_{f}\left(\mathbb{N}\right)$ such that
$y_{1}=\prod_{t\in H_{1}}x_{t}$.

We pick inductively $n\in\mathbb{N}$ and assumptionally chosen $\langle y_{i}\rangle_{i=1}^{n}$
and $\left\langle H_{i}\right\rangle _{i=1}^{n}$ so that : 

(1) we have $y_{i}=\prod_{t\in H_{i}}x_{t}$ for each $i\in\left\{ 1,2,\cdots,n\right\} $.

(2) if $i<n$, then $maxH_{i}<minH_{i+1}$ , and

(3) $FP\left(\langle y_{i}\rangle_{i=1}^{n}\right)\subseteq A^{\,*}$.

Let $K=FP\left(\langle y_{i}\rangle_{i=1}^{n}\right)$ and let $k=max\,H_{n}+1$.
Name 
\[
B=FP\left(\langle x_{i}\rangle_{i=k}^{\infty}\right)\cap A^{\,*}\cap\bigcap_{a\in K}a^{-1}A^{\,*}.
\]

We already know that $a^{-1}A^{\,*}\in p$ for each $a\in K$, so
$B\in p$. Pick $y_{n+1}\in B$ and $H_{n+1}$ so that $min\,H_{n+1}\geq k$
and we get $y_{n+1}=\prod_{t\in H_{n+1}}x_{t}$.

Given $a\in K$, we get $a\cdot y_{n+1}\in A^{\,*}$which gives us
$FP\left(\langle y_{i}\rangle_{i=1}^{n+1}\right)\subseteq A^{\,*}\subseteq A$. 
\end{proof}
Next we state another lemma which follows from these previous two
results.
\begin{lem}
Let a semigroup be $\left(S,\cdot\right)$, and $A$ is a $\mathcal{F-}J$-set
in $S$, let $F\in\mathcal{P}_{f}\left(^{\mathbb{N}}S\right)$ be
$\mathcal{F}-$good and
\[
\Lambda=\left\{ \begin{array}{c}
M\subset\mathbb{N}:M\,\text{is\,finite\,\,and}\,M=\left\{ \tau\left(1\right),\tau\left(2\right),\ldots,\tau\left(m\right)\right\} _{<}\\
\,\text{and\,}\,\exists a\in S^{m+1}\,\text{such\,that}\,\,\forall f\in F\\
\left(a\left(1\right)\cdot f\left(\tau\left(1\right)\right)\cdot a\left(2\right)\ldots a\left(m\right)\cdot f\left(\tau\left(m\right)\right)\cdot a\left(m+1\right)\in A\right)
\end{array}\right\} .
\]
 Let a sequence be $\left\langle G_{n}\right\rangle _{n=1}^{\infty}$
in $\mathcal{P}_{f}\left(\mathbb{N}\right)$ so that $max\,G_{n}<min\,G_{n+1}$
for all $n\in\mathbb{N}$, then we can have a union subsystem $\left\langle H_{n}\right\rangle _{n=1}^{\infty}$
of $\left\langle G_{n}\right\rangle _{n=1}^{\infty}$ for which $FU\left(\left\langle H_{n}\right\rangle _{n=1}^{\infty}\right)\subseteq\Lambda$.
\end{lem}

\begin{proof}
We take an $\mathcal{F}$-\emph{IP }set $FP\left(\langle x_{n}\rangle_{n=1}^{\infty}\right)$
and construct 
\[
T=\bigcap_{m=1}^{\infty}\overline{FP\left(\langle x_{n}\rangle_{n=m}^{\infty}\right)}
\]
 which is a closed subsemigroup of $\bar{\mathcal{F}}$. 

So by Ellis theorem\cite{key-3-1} , we have that $E\left(T\right)\neq\emptyset$,
say $p\in E\left(T\right)$. So for each $m\in\mathbb{N}$, $FP\left(\langle x_{n}\rangle_{n=m}^{\infty}\right)\in p$.
Now if $\bigcup_{i=1}^{r}A_{i}=FP\left(\langle x_{n}\rangle_{n=m}^{\infty}\right)\in p$,
we can choose $i\in\left\{ 1,2,\cdots,r\right\} $ so that $A_{i}\in p$.

Hence by the previous proposition we can choose a product subsystem
$\langle y_{n}\rangle_{n=1}^{\infty}$ of $\langle x_{n}\rangle_{n=1}^{\infty}$,
so that $FP\left(\langle y_{n}\rangle_{n=1}^{\infty}\right)\subseteq A_{i}$. 

For our proof of this lemma we now choose the partition 
\[
A_{1}=\Lambda\cap FU\left(\left\langle G_{n}\right\rangle _{n=1}^{\infty}\right)
\]
 and 
\[
A_{2}=FU\left(\left\langle G_{n}\right\rangle _{n=1}^{\infty}\right)\setminus\Lambda.
\]
As shown earlier, if we pick $A_{i}$, we can have a union subsystem
$\left\langle H_{n}\right\rangle _{n=1}^{\infty}$ for which, $FU\left(\left\langle H_{n}\right\rangle _{n=1}^{\infty}\right)\subseteq A_{i}$.
From earlier lemma we may assume $H\in\mathcal{P}_{f}\left(\mathbb{N}\right)$
for which $\bigcup_{n\in H}H_{n}\in\Lambda$. Since 
\[
\bigcup_{n\in H}H_{n}\in FU\left(\left\langle H_{n}\right\rangle _{n=1}^{\infty}\right)
\]
it is guaranteed that $FU\left(\left\langle H_{n}\right\rangle _{n=1}^{\infty}\right)\subseteq A_{1}\subseteq\Lambda$,
and we are done.
\end{proof}
Now we are in the position to prove the main theorem of this section.
\begin{thm}
Let $\mathcal{F}$ and $\mathcal{G}$ are two idempotent filters on
$S$ and $T$ respectively such that $\bar{\mathcal{F}}$ and $\bar{\mathcal{G}}$
are two closed subsemigroups and $\mathcal{H}$ is a filter on $S\times T$
generated by $\mathcal{F}$ and $\mathcal{G}$ such that $\bar{\mathcal{H}}$
is a closed subsemigroup of $\beta\left(S\times T\right)$. If $A$
is an $\mathcal{F}$-J set and $B$ is a $\mathcal{G}$-J set then
$A\times B$ is an $\mathcal{H}$-J set. 
\end{thm}

\begin{proof}
Let $H\in\mathcal{H}$, then there exists $F\in\mathcal{F}$ and $G\in\mathcal{G}$
such that $F\times G\subseteq H$. 

By definition we know that if $\Gamma$ is an $\mathcal{F}-J$ set,
then for any $\mathcal{F}$ good map say $\Theta\in\mathcal{P}_{f}\left(\,^{\mathbb{N}}S\right)$,
there exists $a\in S^{m+1}$ and $\left\{ h\left(1\right),h\left(2\right),\cdots,h\left(m\right)\right\} _{<}\subset\mathbb{N}$
such that 
\[
a\left(1\right)f\left(h\left(1\right)\right)a\left(2\right)f\left(h\left(2\right)\right)\ldots a\left(m\right)f\left(h\left(m\right)\right)a\left(m+1\right)\in\Gamma
\]
for all $f\in\Theta$, also since $\Theta$ is $\mathcal{F}$ good,
for $F\in\mathcal{F}$ there exists $k\in\mathbb{N}$ such that $ZFP_{k}\left(\Theta\right)\subseteq F$. 

We will start with $\mathcal{H}$ good map $\tilde{H}\in\mathcal{P}_{f}\left(^{\mathbb{N}}S\times T\right)$.
This defines collections 
\[
\hat{F}=\left\{ \begin{array}{c}
f\in\mathcal{P}_{f}\left(^{\mathbb{N}}S\right):f=\pi_{1}\circ h\,\text{where\,}h\in\tilde{H}\,\\
\text{and\,\ensuremath{\pi_{1}} is the projection on the first co-ordinate}
\end{array}\right\} 
\]
 and 
\[
\hat{G}=\left\{ \begin{array}{c}
g\in\mathcal{P}_{f}\left(^{\mathbb{N}}T\right):g=\pi_{2}\circ h,\,\text{where\,}h\in\tilde{H}\,\\
\text{and\,\ensuremath{\pi_{2}} is the projection on the second co-ordinate}
\end{array}\right\} .
\]
 We take these collections as our $\mathcal{F}$ good and $\mathcal{G}$
good maps to work with. 

So, $A$ being an $\mathcal{F}$\emph{-J} set and for $\hat{F}\in\mathcal{P}_{f}\left(\,^{\mathbb{N}}S\right)$
we get desired $a\in S^{m+1}$ and $\left\{ k\left(1\right),k\left(2\right),\cdots,k\left(m\right)\right\} _{<}\subset\mathbb{N}$
so that 
\[
a\left(1\right)f\left(k\left(1\right)\right)a\left(2\right)f\left(k\left(2\right)\right)\ldots a\left(m\right)f\left(k\left(m\right)\right)a\left(m+1\right)\in A.
\]

for all $f\in\hat{F}$ and for the $F\in\mathcal{F}$ opted at the
start of the theorem, there exists $k_{1}\in\mathbb{N}$ such that
$ZFP_{k_{1}}\left(\hat{F}\right)\subseteq F$.

Similarly for $B$, being an $\mathcal{G}$\emph{-J} set and for $\hat{G}\in\mathcal{P}_{f}\left(\,^{\mathbb{N}}T\right)$
we get $b\in T^{m+1}$ and $\left\{ \ell\left(1\right),\ell\left(2\right),\cdots,\ell\left(m\right)\right\} _{<}\subset\mathbb{N}$
so that 
\[
a\left(1\right)f\left(\ell\left(1\right)\right)a\left(2\right)f\left(\ell\left(2\right)\right)\ldots a\left(m\right)f\left(\ell\left(m\right)\right)a\left(m+1\right)\in B.
\]

for all $g\in\hat{G}$ and for given $G\in\mathcal{G}$, there exists
$k_{2}\in\mathbb{N}$ such that $ZFP_{k_{2}}\left(\hat{G}\right)\subseteq G$.

So to start with, for $H\in\mathcal{H}$, take $k=min\left\{ k_{1},k_{2}\right\} \in\mathbb{N}$
such that $ZFP_{k}\left(\tilde{H}\right)=ZFP_{k}\left(\hat{F},\hat{G}\right)\subseteq F\times G\subseteq H$. 

Considering $\hat{F}\in\mathcal{P}_{f}\left(\,^{\mathbb{N}}S\right)$
through lemma 4.5 a union subsystem can be picked, say $\left\langle D_{n}\right\rangle _{n=1}^{\infty}$
of $\left\langle \left\{ n\right\} \right\rangle _{n=1}^{\infty}$
so that 
\[
FU\left(\left\langle D_{n}\right\rangle _{n=1}^{\infty}\right)\subseteq\left\{ \begin{array}{c}
M\in\mathcal{P}_{f}\left(\mathbb{N}\right):M=\left\{ \tau\left(1\right),\tau\left(2\right),\ldots,\tau\left(m\right)\right\} _{<}\\
\text{and\,there\,is}\,\text{a\,\,}a\in S^{m+1}\,\text{\,so\,that\,for\,all\,\,}h\in\tilde{H}\\
a\left(1\right)f\left(\tau\left(1\right)\right)a\left(2\right)f\left(\tau\left(2\right)\right)\ldots a\left(m\right)f\left(\tau\left(m\right)\right)a\left(m+1\right)\in A
\end{array}\right\} .
\]
 Using lemma 4.3 again a $H\in\mathcal{P}_{f}\left(\mathbb{N}\right)$
can be chosen for which 
\[
\bigcup_{n\in H}D_{n}\in\left\{ \begin{array}{c}
M\in\mathcal{P}_{f}\left(\mathbb{N}\right):M=\left\{ \tau\left(1\right),\tau\left(2\right),\ldots,\tau\left(m\right)\right\} _{<}\\
\text{and\,there\,exists}\,\text{a\,\,}b\in T^{m+1}\,\text{so\,that\,for\,all\,\,}h\in\tilde{H}\\
b\left(1\right)g\left(\tau\left(1\right)\right)b\left(2\right)g\left(\tau\left(2\right)\right)\ldots b\left(m\right)g\left(\tau\left(m\right)\right)b\left(m+1\right)\in B
\end{array}\right\} .
\]
 Choose $M$ to be $\bigcup_{n\in H}D_{n}$, so that $M=\left\{ \tau\left(1\right),\tau\left(2\right),\ldots,\tau\left(m\right)\right\} _{<}$,
then we can pick $a\in S^{m+1}$ and $b\in T^{m+1}$ so that for all
$h\in\tilde{H}$, we simultaneously have
\[
a\left(1\right)f\left(\tau\left(1\right)\right)a\left(2\right)f\left(\tau\left(2\right)\right)\ldots a\left(m\right)f\left(\tau\left(m\right)\right)a\left(m+1\right)\in A
\]
 and 
\[
b\left(1\right)g\left(\tau\left(1\right)\right)b\left(2\right)g\left(\tau\left(2\right)\right)\ldots b\left(m\right)g\left(\tau\left(m\right)\right)b\left(m+1\right)\in B.
\]
 Defining $\lambda\left(i\right)=\left(a\left(i\right),b\left(i\right)\right)$,
we have $\lambda\in\left(S\times T\right)^{m+1}$, so for all $h\in\tilde{H}$,
\[
\lambda\left(1\right)h\left(\tau\left(1\right)\right)\lambda\left(2\right)h\left(\tau\left(2\right)\right)\ldots\lambda\left(m\right)h\left(\tau\left(m\right)\right)\lambda\left(m+1\right)\in A\times B.
\]
\end{proof}

\section{Product of some other combinatorially large sets}

In \cite{key-11} Hindman and Strauss showed that the product of two
central sets is again central. Here we want to prove that product
of two strongly central sets is again strongly central. The proof
uses algebra of Stone-\v{C}ech compactification and is the following.
\begin{thm}
Let $S$ and $T$ be two semigroups. Let $A$ be a strongly central
set in $S$ and $B$ be a strongly central set in $T$. Then the cartesian
product $A\times B$ is strongly central set in $S\times T$. 
\end{thm}

\begin{proof}
Let $\tilde{\iota}^{-1}\left(K\left(\beta S\right)\times K\left(\beta T\right)\right)=R$
where $\tilde{\iota}$ is the continuous extension of the identity
function. Then $R$ is a closed subsemigroup of $\beta\left(S\times T\right)$
and $K\left(R\right)=K\left(\beta\left(S\times T\right)\right)\cap R$
by Theorem 1.65 of \cite{key-12}. Let $L$ be a minimal left ideal
of $\beta\left(S\times T\right)$, then $N=L\cap R$ is a minimal
left ideal of $R$. So $\tilde{\iota}\left(N\right)$ is a minimal
left ideal of $\beta S\times\beta T$. Let $L_{1}=\pi_{1}\left(\tilde{\iota}\left(N\right)\right)$
and $L_{2}=\pi_{2}\left(\tilde{\iota}\left(N\right)\right)$ respectively
be the coordinate wise projections. Then $L_{1}$ is a minimal left
ideal of $\beta S$ and $L_{2}$ is a minimal left ideal of $\beta T$.
Since $A$ is strongly central in $S$ and $B$ is strongly central
in $T$, there exists $p\in L_{1}\cap\bar{A}$ and $q\in L_{2}\cap\bar{B}$.
Denote $M=\tilde{\iota}^{-1}\left[\left\{ \left(p,q\right)\right\} \right]$
and by Theorem 4.43.1 of \cite{key-12} pick an idempotent $r\in K\left(M\right)$
such that $\tilde{\iota}\left(r\right)=\left(p,q\right)$. Then $A\times B\in r$
since $A\in p$ and $B\in q$. Also $r\in N\subseteq L$. Therefore
$r\in L\cap\overline{\left(A\times B\right)}$ which concludes the
proof. 
\end{proof}
There is another alternative method of proving the previous theorem,
with the help of the following combinatorial characterization of right
strongly central sets, shown in \cite{key-4-1}.
\begin{lem}
Let $S$ be a semigroup and $A\subset S$. Then $A$ is right strongly
Central if and only if whenever family $\mathcal{A}$ of subsets of
$S$ with right thick finite intersection property, there exists a
downward directed family $\left\langle C_{F}\right\rangle _{F\in I}$
of subsets of $A$ such that

(i) for each $F\in I$ and each $x\in C_{F}$ there exists $G\in I$
with $C_{G}\subseteq x^{-1}C_{F}$ and

(ii) $\mathcal{A}\cup\left\{ C_{F}:F\in I\right\} $ has the finite
intersection property.
\end{lem}

\begin{proof}
See \cite[Theorem 2.6]{key-4-1}.
\end{proof}
We quickly prove another small lemma, which is the following.
\begin{lem}
Let $S$ and $T$ be a semigroup and $C\subseteq S\times T$ be a
thick set. Then $A=\pi_{1}\left(C\right)$ and $B=\pi_{2}\left(C\right)$
are thick sets in $S$ and $T$ respectively.
\end{lem}

\begin{proof}
Without loss of generality , let $A$ is not thick in $S$. Then there
exists a finite subset $H_{1}$ of $S$ such that for any $x\in S$,
$H_{1}\cdot x$ is not contained in $A$. Now, for any finite subset
$H_{2}$ of $T$, 
\[
H_{1}\times H_{2}\in P_{f}\left(S\times T\right)
\]
 and for any $\left(x,y\right)\in S\times T$, $\left(H_{1}\times H_{2}\right)\left(x,y\right)$
is not contains in $C$, a contradiction to the fact that $C$ is
thick.

So, $A=\pi_{1}\left(C\right)$ and $B=\pi_{2}\left(C\right)$ are
thick sets in $S$ and $T$ respectively.
\end{proof}
We now give a definiton which was first introduced in \cite{key-4-1}
which is a special type of finite intersection property and this will
be used in the next two theorems.
\begin{defn}
\textit{\emph{Let $\mathcal{A}$ be a family of subsets of a semigroup
$\left(S,+\right)$. Then $A$ has the }}\textit{thick finite intersection
property}\textit{\emph{ if and only if any intersection of finitely
many members of $\mathcal{A}$ is thick.}}\textit{ }
\end{defn}

Now we prove the following.
\begin{thm}
Let $S$ and $T$ be semigroups and $A\subset S,B\subset T$ be right
strongly Central sets in $S$ and $T$ respectively. Then $A\times B$
is right strongly Central set in $S\times T$.
\end{thm}

\begin{proof}
Let, $\mathcal{C}$ be a family of subsets of $S\times T$ with thick
finite intersection property. Let us consider $\mathcal{A}=\left\{ \pi_{1}\left(C\right):C\in\mathcal{C}\right\} $and
$\mathcal{B}=\left\{ \pi_{2}\left(C\right):C\in\mathcal{C}\right\} $
be two families of subsets of $S$ and $T$ respectively.

Both of them has thick finite intersection property by the previous
lemma and the fact that finite intersection of thick sets is thick.

Now $A$ is right strongly central set in $S$, then for the family
$\mathcal{A}$ of subsets of $S$ with right thick finite intersection
property, there exists a downward directed family $\left\langle C_{F}\right\rangle _{F\in I}$
of subsets of $S$ such that 

(i) for each $F\in I$ and each $x\in C_{F}$ there exists $G\in I$
with $C_{G}\subseteq x^{-1}C_{F}$ and

(ii) $\mathcal{A}\cup\left\{ C_{F}:F\in I\right\} $ has the finite
intersection property. 

Also , $B$ is right strongly central set in $T$, then for the family
$\mathcal{B}$ of subsets of $T$ with right thick finite intersection
property, there exists a downward directed family $\left\langle D_{F}\right\rangle _{F\in I}$
of subsets of $T$ such that 

(i) for each $F\in I$ and each $y\in D_{F}$ there exists $G\in I$
with $D_{G}\subseteq y^{-1}D_{F}$ and

(ii) $\mathcal{A}\cup\left\{ D_{F}:F\in I\right\} $ has the finite
intersection property. 

Then $\left\langle C_{F},D_{F}\right\rangle _{F\in I}$ is the downward
directed family of $\left(A\times B\right)$ and 

(i) for each $F\in I$ and each $\left(x,y\right)\in E_{F}$ there
exists $G\in I$ with $E_{G}\subseteq\left(x^{-1},y^{-1}\right)E_{F}=\left(x,y\right)^{-1}E_{F}$
and

(ii) $\left(\mathcal{A}\times\mathcal{B}\right)\cup\left\{ E_{F}:F\in I\right\} $
has the finite intersection property. 

So $\mathcal{C}\cup\left\{ E_{F}:F\in I\right\} $ has the finite
intersection property. 
\end{proof}
Another notion, called right thickly central was also defined in \cite{key-4-1}
where the authors characterised the sets combinatorially. Here we
show that product of two such sets is again a set of this kind, using
the theorem proved by the authors in their paper, which is the following.
\begin{thm}
Let $S$ be a semigroup and $A\subset S$. Then $A$ is right thickly
Central if and only if there is a family $\mathcal{A}$ of subsets
of $S$ with right thick finite intersection property, and for any
downward directed family $\left\langle C_{F}\right\rangle _{F\in I}$
of subsets of $S\setminus A$ such that

(i) for each $F\in I$ and each $x\in C_{F}$ there exists $G\in I$
with $C_{G}\subseteq x^{-1}C_{F}$ and

(ii) $\mathcal{A}\cup\left\{ C_{F}:F\in I\right\} $ does not have
finite intersection property.
\end{thm}

\begin{proof}
\cite[Theorem 2.3]{key-4-1}.
\end{proof}
Hence we have the following one.
\begin{thm}
Let $S$ and $T$ be two semigroups. Let $A$ be a right thickly central
set in $S$ and $B$ be a right thickly central set in $T$. Then
the cartesian product $A\times B$ is right thickly central set in
$S\times T$. 
\end{thm}

\begin{proof}
Let $\left\langle E_{F}\right\rangle _{F\in I}$ be a downward directed
family of $\left(S\times T\right)\setminus\left(A\times B\right)$
and let for any $F\in I$, $C_{F}=\pi_{1}\left(E_{F}\right)$ and
$D_{F}=\pi_{2}\left(E_{F}\right)$ are subsets of $S\setminus A$
and $T\setminus B$ respectively. 

$A$ is right thickly central set in $S$, so there is a family $\mathcal{A}$
of subsets of $S$ with right thick finite intersection property.
Then for the downward directed family $\left\langle C_{F}\right\rangle _{F\in I}$
of subsets of $S\setminus A$ we have 

(i) for each $F\in I$ and each $x\in C_{F}$ there exists $G_{1}\in I$
with $C_{G_{1}}\subseteq x^{-1}C_{F}$ and

(ii) $\mathcal{A}\cup\left\{ C_{F}:F\in I\right\} $ does not have
finite intersection property. 

Also , $B$ is right thickly central set in $T$, so there is a family
$\mathcal{B}$ of subsets of $T$ with right thick finite intersection
property, and hence there exists a downward directed family $\left\langle D_{F}\right\rangle _{F\in I}$
of subsets of $T\setminus B$ such that 

(i) for each $F\in I$ and each $y\in D_{F}$ there exists $G_{2}\in I$
with $D_{G_{2}}\subseteq y^{-1}D_{F}$ and

(ii) $\mathcal{B}\cup\left\{ D_{F}:F\in I\right\} $ does not have
finite intersection property. 

We want to show $\mathcal{A}\times\mathcal{B}$ is the required family
of subsets of $S\times T$. We have that

(i) for each $F\in I$ and each $\left(x,y\right)\in E_{F}$ there
exists $G\in I$ where $G=max\left(G_{1},G_{2}\right)$ with $E_{G}\subseteq\left(x^{-1},y^{-1}\right)E_{F}=\left(x,y\right)^{-1}E_{F}$
and

(ii) $\left(\mathcal{A}\times\mathcal{B}\right)\cup\left\{ E_{F}:F\in I\right\} $
does not have finite intersection property. 

Hence we have the required result.
\end{proof}
Another alternative way is there to prove the last fact, which we
mention as a passing argument.

Since $A\subset S,B\subset T$ are right thickly central sets respectively
in $S$ and $T$, then $A^{c}\subset S,B^{c}\subset T$ are not right
strongly central sets. So $A^{c}\times B^{c}=\left(A\times B\right)^{c}$
is right strongly central set in $S\times T$. Therefore $A\times B$
is right thickly Central set in $S\times T$.

$\vspace{0.3in}$

\textbf{Acknowledgment:} The second author acknowledges the Grant
CSIR-UGC NET fellowship with file No. 09/106(0199)/2019-EMR-I.

$\vspace{0.3in}$

\end{document}